\documentclass{birkjour}
\usepackage[mathscr]{eucal}

\newtheorem{theorem}{Theorem}[section]
\newtheorem{lemma}[theorem]{Lemma}
\newtheorem{proposition}[theorem]{Proposition}
\theoremstyle{definition}

\newtheorem{example}[theorem]{Example}

\newtheorem{cor}[theorem]{Corollary}

\theoremstyle{remark}
\newtheorem{remark}[theorem]{Remark}

\numberwithin{equation}{section}

\begin{document}

		\title[Min-max relations for tuples of operators in terms of component spaces] {{Min-max relations for tuples of operators in terms of component spaces}}
	
	
\author[Arpita Mal]{Arpita Mal}

\address{Dhirubhai Ambani University\\ Gandhinagar-382007\\ India.}
\email{arpitamalju@gmail.com}






\date{}

\begin{abstract}
For tuples of compact operators $\mathcal{T}=(T_1,\ldots, T_d)$ and $\mathcal{S}=(S_1,$ $\ldots,S_d)$ on Banach spaces over a field $\mathbb{F}$, considering the joint $p$-operator norms on the tuples, we study $dist(\mathcal{T},\mathbb{F}^d\mathcal{S}),$ the distance of $\mathcal{T}$ from the $d$-dimensional subspace $\mathcal{F}^d\mathcal{S}:=\{\textbf{z}\mathcal{S}:\textbf{z}\in \mathbb{F}^d\}.$ We obtain a relation between $dist(\mathcal{T},\mathbb{F}^d\mathcal{S})$ and $dist(T_i,\mathbb{F}S_i),$ for $1\leq i\leq d.$ We prove that if $p=\infty,$ then $dist(\mathcal{T},\mathbb{F}^d\mathcal{S})=\underset{1\leq i\leq d}{\max}dist(T_i,\mathbb{F}S_i),$ and for $1\leq p<\infty,$ under a sufficient condition, $dist(\mathcal{T},\mathbb{F}^d\mathcal{S})^p=\underset{1\leq i\leq d}{\sum}dist(T_i,\mathbb{F}S_i)^p.$  As a consequence, we deduce the equivalence of Birkhoff-James orthogonality,  $\mathcal{T}\perp_B \mathbb{F}^d\mathcal{S} \Leftrightarrow  T_i\perp_B S_i,$ under a sufficient condition. Furthermore, we explore the relation of one sided G\^ateaux derivatives of $\mathcal{T}$ in the direction of $\mathcal{S}$
with that of $T_i$ in the direction of $S_i.$ Applying this, we explore the relation between the smoothness of $\mathcal{T}$ and $T_i.$  By identifying an operator, whose range is $\ell_\infty^d,$ as a tuple of functionals, we effectively use the results obtained here for operators whose range is $\ell_\infty^d$ and deduce nice results involving functionals. 
\end{abstract}

\maketitle

\section{Introduction}
The purpose of this article is to explore the relation between the distance of tuples of operators and the distance of its components. Let us first introduce the notations and terminologies that will be used throughout the article.\\
The symbols $\mathcal{X}, \mathcal{Y}, \mathcal{Y}_i~(1\leq i\leq d)$ denote Banach spaces and $\mathcal{H}$ denotes a Hilbert space over the field $\mathbb{F}$, where $\mathbb{F}=\mathbb{R}$ or $\mathbb{C}.$  
Let $B_{\mathcal{X}},$ $S_\mathcal{X}$ denote the closed unit ball and the unit sphere of $\mathcal{X},$ respectively. 
For $x\in \mathcal{X}$ and a subspace $W$ of $\mathcal{X},$ the distance of $x$ from $W$ is defined as $$dist(x,W):=\inf\{\|x-w\|:w\in W\}.$$ 
An element $w_0\in W$ is said to be a best approximation to $x$ out of $W,$ if $\|x-w_0\|=dist(x,W).$ Let 
$$P_W(x):=\{w_0\in W:\|x-w_0\|=dist(x,W)\}.$$ 
In general, $P_W(x)$ may be empty. However, if either $W$ is finite-dimensional or $\mathcal{X}$ is reflexive and $W$ is closed, then $P_W(x)\neq \emptyset$ for all $x\in \mathcal{X}.$ Birkhoff-James (B-J) orthogonality and best approximation are equivalent notions. For $x,y\in \mathcal{X},$ $x$ is said to be B-J orthogonal \cite{B,J} to $y,$ denoted as $x\perp_B y,$ if 
$$\|x+\lambda y\|\geq \|x\|,\quad \text{ for all scalars } \lambda.$$ 
  We say that $x\perp_B W$ if $x\perp_B w$ for all $w\in W.$ Observe that $w_0\in P_W(x)$ if and only if $x-w_0\perp_B W.$ Let $\mathcal{X}^*$ denote the dual space of $\mathcal{X}.$ 
 We denote  the set of all extreme points of a convex set $C$ by $E_C.$ For simplicity, we denote  $E_{B_{\mathcal{X}}}$ by $E_{\mathcal{X}}.$ 
For $(0\neq)x\in \mathcal{X},$ let $J(x):=\{f\in S_{\mathcal{X}^*}:f(x)=\|x\|\}.$ Note that, $J(x)$ is a non-empty, convex, weak*compact, extremal subset of $B_{\mathcal{X}^*}.$ If $J(x)$ is singleton, then $x$ is said to be smooth in $\mathcal{X}.$ It is well-known that $x$ is smooth if and only if the norm of $\mathcal{X}$ is G\^ateaux differentiable at $x.$ Recall that 
\[\rho_+(x,y)=\lim_{t\to0+}\frac{\|x+ty\|-\|x\|}{t}, ~\text{ and } \rho_-(x,y)=\lim_{t\to0-}\frac{\|x+ty\|-\|x\|}{t} \]
are called, respectively, the right-hand and left-hand G\^ateaux derivative of $x$ in the direction of $y.$ Note that in general, $\rho_-(x,y)\leq \rho+(x,y).$ If $\rho_+(x,y)=\rho_-(x,y)$ holds for all $y,$ then the norm is said to be G\^ateaux differentiable at $x.$ It is well-known from \cite{DGZ,HL} that
\begin{align*}
	\rho_+(x,y)= \sup\{f(y):f\in E_{J(x)}\}, \text{ and }	\rho_-(x,y)= \inf\{f(y):f\in E_{J(x)}\}.
	\end{align*}
See \cite{GI} for more results on the mapping $P_W$ and these one sided limits. The readers may follow \cite{CL} for more results on approximation theory.
We reserve the notation $\ell_p^d(\mathcal{Y}_k)$ to denote $\underset{1\leq k\leq d}{\oplus}\mathcal{Y}_k,$  the direct sum of $\mathcal{Y}_k,$ where  for $y_k\in \mathcal{Y}_k,$
\[\|(y_1,y_2,\ldots,y_d)\|_p=\begin{cases}
	\bigg(\displaystyle{\sum_{i=1}^d}\|y_i\|^p\bigg)^\frac{1}{p},~\text{ if }1\leq p<\infty\\
	\displaystyle{\max_{1\leq i\leq d}}\|y_i\|, \text{ if } p=\infty
\end{cases}.\]
Note that $(\ell_p^d(\mathcal{Y}_k))^*=\ell_q^d(\mathcal{Y}_k^*),$ where as usual for  $1< p<\infty,$ $\frac{1}{p}+\frac{1}{q}=1,$ for $p=1,$  $q=\infty,$ and for $p=\infty,$ $q=1.$ 
Let $\mathcal{L}(\mathcal{X,\mathcal{Y}})$ (resp., $\mathcal{K}(\mathcal{X,\mathcal{Y}})$) be the space of all bounded (resp., compact) linear operators from $\mathcal{X}$ to $\mathcal{Y}$. For $T_k\in \mathcal{L}(\mathcal{X,\mathcal{Y}}_k),$ suppose $\mathcal{T}$ denotes the $d$-tuple $(T_1,T_2,\ldots,T_d).$ Then $\mathcal{T}:\mathcal{X}\to\ell_p^d(\mathcal{Y}_k)$ defined as $\mathcal{T}x=(T_1x,T_2x,\ldots,T_dx)$ is a bounded linear operator. Moreover, $\mathcal{T}$ is compact, if each $T_k$ is compact. Consider the usual operator norm of $\mathcal{T}\in \mathcal{L}(\mathcal{X},\ell_p^d(\mathcal{Y}_k)).$  Observe that $\mathcal{T}^*:\ell_q^d(\mathcal{Y}_k^*)\to \mathcal{X}^*,$ where  for all $f_k\in \mathcal{Y}_k^*,$ $\mathcal{T}^*(f_1,\ldots,f_d)=\sum_{k=1}^dT_k^*f_k.$
Distance formulae and B-J orthogonality of operators have been extensively studied  in the literature (see \cite{BS,C,G,GS3,M,MP22,MP,RAO21,Sb,SMP,SPM2} for some references). Interested readers may follow the recent book chapter \cite{AGKZ} and the monograph \cite{MPS}  for more information in this topic.
 Recently, in \cite{GS, GS2} Grover and Singla studied B-J orthogonality for tuples of operators defined on Hilbert spaces for the special case $p=2.$  Suppose $\mathcal{S}=(S_1,\ldots,S_d), $ where $S_k\in \mathcal{L}(\mathcal{X},\mathcal{Y}_k)$ and for $\textbf{z}=(z_1,\ldots,z_d)\in \mathbb{F}^d,$ $\textbf{z}\mathcal{S}=(z_1S_1,\ldots,z_dS_d).$ Throughout the article, $\mathcal{T}^0$ denotes the tuple $\mathcal{T}-\textbf{z}^0\mathcal{S},$ where $\textbf{z}^0=(z_1^0,\ldots,z_d^0)\in \mathbb{F}^d$ and $T_j^0=T_j-z_j^0S_j.$ Motivated by \cite{GS,GS2}, we study $dist(\mathcal{T},\mathbb{F}^d\mathcal{S}).$ The following questions arise naturally.
\begin{enumerate}
\item 	Is there any relation between $dist(\mathcal{T},\mathbb{F}^d\mathcal{S})$ and $dist(T_i,\mathbb{F}S_i),$ for some $1\leq i\leq d?$

\item How are the notions $\mathcal{T}\perp_B \mathcal{S}$ and $T_i\perp_B S_i$ related?

\item How are the one sided G\^ateaux derivatives of tuples of operators related to that of its components? 

\item How are the smoothness of $\mathcal{T}$ and the smoothness of $T_i$  related?
\end{enumerate}
In this article, we address these questions and explore the relations.
 For an operator $T\in \mathcal{L}(\mathcal{X},\mathcal{Y}),$ we denote the norm attainment set of $T$ by $M_T:=\{x\in S_{\mathcal{X}}:\|Tx\|=\|T\|\}.$ It is well-known that a compact operator $T$ on a reflexive Banach space is smooth if and only if there exists a unit vector $x$ such that $M_T=\{\alpha x:\alpha\in \mathbb{F}, |\alpha|=1\}$ and $Tx$ is smooth (see \cite{MPS}). Observe that for each $T\in \mathcal{K}(\mathcal{X},\ell_p^d),$ there exist functionals $f_1,\ldots,f_d\in \mathcal{X}^*$ such that $Tx=(f_1(x),\ldots,f_d(x))$ for all $x\in \mathcal{X}.$ Clearly, $T$ can be considered as a tuple of functionals on $\mathcal{X}.$ Applying the results for tuples of operators, we get some interesting distance formula and equivalence of B-J orthogonality of operators in $\mathcal{K}(\mathcal{X},\ell_\infty^d),$ which also illustrate the applicability of the results in this article.\\
To prove the desired results, we frequently use the extreme points  of $B_{\mathcal{K}(\mathcal{X},\mathcal{Y})^*}.$ From \cite[Th. 1.3]{RS}, we note that 
\begin{equation}\label{eq-rsg}
	E_{\mathcal{K}(\mathcal{X},\mathcal{Y})^*}=\{x^{**}\otimes y^*:x^{**}\in E_{\mathcal{X}^{**}},y^*\in E_{\mathcal{Y}^*}\}, 
\end{equation}
where $x^{**}\otimes y^*(S)=x^{**}(S^*y^*)$ for $S\in \mathcal{K}(\mathcal{X},\mathcal{Y}).$ In particular, if $\mathcal{X}$ is reflexive, then
\begin{equation}\label{eq-rs}
	E_{\mathcal{K}(\mathcal{X},\mathcal{Y})^*}=\{y^{*}\otimes x:x\in E_{\mathcal{X}},y^*\in E_{\mathcal{Y}^*}\}, 
\end{equation}
where $y^{*}\otimes x(S)=y^{*}(Sx)$ for $S\in \mathcal{K}(\mathcal{X},\mathcal{Y}).$ The other main tool of this article is 	the following fundamental characterization of best approximation due to Singer.
\begin{theorem}\cite[Th. 1.1, pp. 170]{S}\label{th-singer}
	Let $x\in \mathcal{X},$ and $W$ be a subspace of $\mathcal{X}$ such that $\dim(W)=n$ and $x\notin W.$ Then $y\in P_{W}(x)$ if and only if there exist
	$h$ extreme points $f_1,f_2,\ldots,f_h\in E_{\mathcal{X}^*},$ ($h\leq n+1$ if $\mathbb{F}=\mathbb{R}$ and $h\leq 2n+1$ if $\mathbb{F}=\mathbb{C}$), $h$ numbers $t_1,t_2,\ldots,t_h>0$ such that $\sum_{i=1}^h t_i=1,$ $\sum_{i=1}^h t_if_i(w)=0$ for all $w\in W$ and $f_i(x-y)=\|x-y\|$ for all $1\leq i\leq h.$
\end{theorem}
In Section 2,  we prove that 
\[dist(\mathcal{T},\mathbb{F}^d\mathcal{S})=\begin{cases}
	\underset{1\leq i\leq d}{\max}dist(T_i,\mathbb{F}S_i), \quad \text{if } p=\infty\\
	\bigg(\underset{1\leq i\leq d}\sum dist(T_ix,\mathbb{F}S_ix)^p\bigg)^\frac{1}{p}, \text{ for some } x\in E_{\mathcal{X}}, \text{ if } 1\leq p <\infty,
	\end{cases}\]
where in the latter case, we assume smoothness of a tuple.
We obtain a sufficient condition to get $dist(\mathcal{T},\mathbb{F}^d\mathcal{S})^p=\underset{1\leq i\leq d}\sum dist(T_i,\mathbb{F}S_i)^p,$ for $1\leq p<\infty.$ 
 As a consequence, we deduce that under a sufficient condition,  if $p=\infty,$ then  $\mathcal{T}\perp_B \mathbb{F}^d\mathcal{S} \Leftrightarrow  T_i\perp_B S_i$  for some  $1\leq i\leq d,$ whereas   if  $1\leq p<\infty,$ then  $\mathcal{T}\perp_B \mathbb{F}^d\mathcal{S} \Leftrightarrow  T_i\perp_B S_i$  for all  $1\leq i\leq d.$ We effectively use these results on $\mathcal{K}(\mathcal{X},\ell_\infty^d)$ and prove that if $T,S\in \mathcal{K}(\mathcal{X},\ell_\infty^d),$ then  $dist(T,\mathbb{F}^dS)=\underset{1\leq i\leq d}{\max}\|f_i|_{\ker(g_i)}\|,$ where $T=(f_1,\ldots,f_d), ~S=(g_1,\ldots,g_d)$ for $f_i,g_i\in \mathcal{X}^*.$  In the direction of one sided G\^ateaux derivatives, we first answer the third question raised earlier and then using it we explore sufficient and necessary conditions for smoothness of $\mathcal{T}$ in terms of smoothness of $T_i$. 

\section{Main results}

We begin the with the  case $1\leq p<\infty.$ Note that, if  $\mathcal{T},\mathcal{S}\in \mathcal{K}(\mathcal{X},\ell_{p}^d(\mathcal{Y})),$ then $\sum_{j=1}^d\|T_jx\|^p\leq \|\mathcal{T}x\|^p$ for each  $x\in S_{\mathcal{X}},$ which implies that $\sum_{j=1}^d\|T_jx\|^p\leq \|\mathcal{T}\|^p.$ Thus, $$\underset{1\leq j\leq d}{\sum}dist(T_jx,\mathbb{F}S_jx)^p\leq dist(\mathcal{T},\mathbb{F}^d\mathcal{S})^p,$$ for each $x\in S_{\mathcal{X}}.$ In the next proposition, we provide a sufficient condition for the equality here for some unit vector $x.$

\begin{proposition}\label{prop-04}
	Suppose $\mathcal{X}$ is reflexive. Let $\mathcal{T},\mathcal{S}\in \mathcal{K}(\mathcal{X},\ell_p^d(\mathcal{Y}_k)),$ where $\mathcal{T}\notin \mathbb{F}^d\mathcal{S}.$ Suppose $dist(\mathcal{T},\mathbb{F}^d\mathcal{S})=\|\mathcal{T}^0\|,$ and $\mathcal{T}^0$ is smooth. Then there exists $x\in M_{\mathcal{T}^0}\cap E_{\mathcal{X}}$ such that
	\[dist(T_jx,\mathbb{F}S_jx)=\|T_j^0x\| ~\text{ for all }1\leq j\leq d,\]
	and 
\begin{equation}\label{eq-012}
	dist(\mathcal{T},\mathbb{F}^d\mathcal{S})^p=\sum_{j=1}^ddist(T_jx,\mathbb{F}S_jx)^p.
\end{equation}
\end{proposition}
\begin{proof}
	 Note that $dist(\mathcal{T},\mathbb{F}^d\mathcal{S})=\|\mathcal{T}^0\|$ implies that $\mathcal{T}^0\perp_B \mathbb{F}^d\mathcal{S}.$ Since $\mathcal{T}^0$ is smooth, there exists a unit vector $x$  such that $M_{\mathcal{T}^0}=\{\alpha x:\alpha\in \mathbb{F},|\alpha|=1\}.$ It is easy to check that $x\in E_{\mathcal{X}}.$ Moreover, from \cite[Ch. 6]{MPS} it follows that  $\mathcal{T}^0x\perp_B \mathbb{F}^d\mathcal{S}x.$ Hence, 
	 \begin{align*}
	& \|(T_1^0x,\ldots,T_d^0x)+(z_1S_1x,\ldots,z_dS_jx)\|^p\geq\|\mathcal{T}^0x\|^p, ~\text{ for all } (z_1,\ldots,z_d)\in \mathbb{F}^d\\
	 	\Rightarrow& \sum_{j=1}^d\|T_j^0x+z_jS_jx\|^p\geq \sum_{j=1}^d\|T_j^0x\|^p , ~\text{ for all } (z_1,\ldots,z_d)\in \mathbb{F}^d\\
	 	\Rightarrow& \|T_j^0x+zS_jx\|^p\geq \|T_j^0x\|, ~\text{ for all } z\in \mathbb{F}, ~1\leq j\leq d\\
	 	\Rightarrow& T_j^0x\perp_B S_jx,  ~\text{ for all }1\leq j\leq d\\
	 	\Rightarrow& dist(T_jx,\mathbb{F}S_jx)=\|T_j^0x\| ~\text{ for all }1\leq j\leq d.
	 	\end{align*}
	 	Therefore, 
	\[dist(\mathcal{T},\mathbb{F}^d\mathcal{S})^p=\|\mathcal{T}^0\|^p=\|\mathcal{T}^0x\|^p=\sum_{j=1}^d\|T_j^0x\|^p=\sum_{j=1}^ddist(T_jx,\mathbb{F}S_jx)^p.\]
\end{proof}

Observe that for each unit vector $x,$ $dist(T_jx,\mathbb{F}S_jx)\leq dist(T_j,\mathbb{F}S_j).$ Therefore,
from (\ref{eq-012}) it clearly follows that if $\mathcal{T}^0$ is smooth, then
\begin{equation}\label{eq-013}
	dist(\mathcal{T},\mathbb{F}^d\mathcal{S})^p\leq \sum_{j=1}^ddist(T_j,\mathbb{F}S_j)^p.
\end{equation}
The natural question that arises now is for the equality condition in (\ref{eq-013}). We address this question now. We  use the following lemma to answer the question. The proof of the lemma is trivial. For the sake of completeness, we provide the proof here.

\begin{lemma}\label{lem-02}
	Suppose $\mathcal{T}\in \mathcal{K}(\mathcal{X},\ell_{p}^d(\mathcal{Y}_k)).$  Let $\cap_{i=1}^dM_{T_i}\neq \emptyset.$ Then $$M_{\mathcal{T}}=\cap_{i=1}^dM_{T_i} \quad \text{ and }\quad \|\mathcal{T}\|^p=\sum_{i=1}^d\|T_i\|^p.$$
\end{lemma}
\begin{proof}
We first prove that $\cap_{i=1}^dM_{T_i}\subseteq M_{\mathcal{T}}.$	Let $x\in \cap_{i=1}^dM_{T_i},$ that is, $\|T_ix\|=\|T_i\|$ for all $1\leq i\leq d.$ So
	\begin{equation}\label{eq-01}
		\|\mathcal{T}\|^p\geq \|\mathcal{T}x\|^p=\sum_{i=1}^d\|T_ix\|^p=\sum_{i=1}^d\|T_i\|^p.
		\end{equation}
	On the other hand, since for each $z\in S_{\mathcal{X}},$
		\[\|\mathcal{T}z\|^p=\sum_{i=1}^d\|T_iz\|^p\leq \sum_{i=1}^d\|T_i\|^p,\] and so \[\|\mathcal{T}\|^p=\sup_{z\in S_{\mathcal{X}}}\|\mathcal{T}z\|^p\leq  \sum_{i=1}^d\|T_i\|^p. \]
		Therefore, the first inequality in (\ref{eq-01}) is equality and we have $\|\mathcal{T}\|=\|\mathcal{T}x\|,$ that is $x\in M_{\mathcal{T}}.$  \\
		Now to prove $M_{\mathcal{T}}\subseteq \cap_{i=1}^dM_{T_i},$ choose $x\in M_{\mathcal{T}}.$ Then 
		\[\|\mathcal{T}\|^p=\|\mathcal{T}x\|^p=\sum_{i=1}^d\|T_ix\|^p\leq \sum_{i=1}^d\|T_i\|^p=\|\mathcal{T}\|^p,\] by the equality of (\ref{eq-01}). Therefore, $\|T_ix\|=\|T_i\|$ for each $1\leq i\leq d,$ that is, $x\in \cap_{i=1}^dM_{T_i}.$
\end{proof}

In the next theorem, we provide a sufficient condition for the equality in (\ref{eq-013}). Later on we provide examples of tuples satisfying the sufficient condition. We prove the theorem assuming $\mathcal{Y}_k=\mathcal{Y}$ for all $1\leq k\leq d.$ The same proof holds in the general case with some minor modification in notation. However, to avoid complexity of notation, we are considering this simpler case in the following theorem as well as in Theorem \ref{th-07}, Corollary \ref{cor-03} and Corollary \ref{cor-01}.
\begin{theorem}\label{th-02}
	Suppose $\mathcal{X}$ is reflexive. Let $\mathcal{T},\mathcal{S}\in \mathcal{K}(\mathcal{X},\ell_p^d(\mathcal{Y})),$  where $\mathcal{T}\notin \mathbb{F}^d\mathcal{S},$ and $dist(\mathcal{T},\mathbb{F}^d\mathcal{S})=\|\mathcal{T}^0\|.$  Suppose $\cap_{j=1}^dM_{T_j^0}\neq \emptyset.$ Then 
\begin{equation}\label{eq-02}
dist(T_j,\mathbb{F}S_j)=\|T_j^0\|~\text{ for all } 1\leq j\leq d~\text{ and}
\end{equation}
	\begin{equation}\label{eq-07}
	dist(\mathcal{T},\mathbb{F}^d\mathcal{S})^p=\sum_{j=1}^ddist(T_j,\mathbb{F}S_j)^p.
	\end{equation}
\end{theorem}
\begin{proof}
	Since $\cap_{j=1}^dM_{T_j^0}\neq \emptyset,$ from Lemma \ref{lem-02} it follows that 
	\[M_{\mathcal{T}^0}=\cap_{j=1}^dM_{T_j^0} \quad \text{ and }\quad \|\mathcal{T}^0\|^p=\sum_{j=1}^d\|T_j^0\|^p.\]
	Since $dist(\mathcal{T},\mathbb{F}^d\mathcal{S})=\|\mathcal{T}^0\|$ implies that $\textbf{z}^0\mathcal{S}\in P_{\mathbb{F}^d\mathcal{S}}(\mathcal{T}),$ so by Theorem \ref{th-singer} and (\ref{eq-rs}), there exist $f_i\in E_{(\ell_{p}^d(\mathcal{Y}))^*},$  $x_i \in E_{\mathcal{X}}$ and  numbers $t_i>0$ for $1\leq i\leq h,$  such that $$\sum_{i=1}^h t_i=1,~ \sum_{i=1}^h t_if_i\otimes x_i(\textbf{z}\mathcal{S})=0~\forall ~\textbf{z}\in \mathbb{F}^d, \text{ and } f_i\otimes x_i(\mathcal{T}^0)=\|\mathcal{T}^0\|,~ \forall ~1\leq i\leq h. $$
	Since $(\ell_{p}^d(\mathcal{Y}))^*=\ell_{q}^d(\mathcal{Y}^*),$ for each $1\leq i \leq h,$ there exists $f_{ij}\in \mathcal{Y}^*$ for $1\leq j\leq d,$ such that $f_{i}=(f_{i1},\ldots,f_{id})\in \ell_{q}^d(\mathcal{Y}^*).$ Clearly, $\|f_i\|_q=1.$ It is easy to observe that if $f_{ij}\neq 0, $ then $\frac{1}{\|f_{ij}\|}f_{ij}\in E_{\mathcal{Y}^*}$ and so $\frac{1}{\|f_{ij}\|}f_{ij}\otimes x_i\in E_{\mathcal{K}(\mathcal{X},\mathcal{Y})^*}.$
	Using H\"{o}lder's inequality, for each $1\leq i\leq h,$ we get 
	\begin{equation}\label{eq-09}
		\begin{split}
			\|\mathcal{T}^0\|= f_i\otimes x_i(\mathcal{T}^0)
			&= f_i(\mathcal{T}^0(x_i))\\
			&	=\sum_{j=1}^d f_{ij}(T_j^0x_i)\\
			&	\leq \sum_{j=1}^d \|f_{ij}\|\|T_j^0x_i\|\\
			&	\leq\|f_{i}\|_q    \bigg(  \sum_{j=1}^d \|T_j^0x_i\|^p    \bigg)^{\frac{1}{p}}\\
			&	= \|\mathcal{T}^0x_i\|\leq \|\mathcal{T}^0\|.
		\end{split}
	\end{equation}
	So equality holds in the above inequalities. Thus, $\|\mathcal{T}^0x_i\|=\|\mathcal{T}^0\|,$ that is $x_i\in M_{\mathcal{T}^0},$  for all $1\leq i\leq h.$ Since $M_{\mathcal{T}^0}=\cap_{j=1}^dM_{T_j^0},$ we have $\|T_j^0x_i\|=\|T_j^0\|.$ Moreover, 
	\begin{equation}\label{eq-08}
		f_{ij}(T_j^0x_i)=\|f_{ij}\|\|T_j^0x_i\| 
		\quad \Rightarrow \quad f_{ij}\otimes x_i(T_j^0)=\|f_{ij}\|\|T_j^0\|.
	\end{equation}
	For $p=1,$ the equality in (\ref{eq-09}) shows that $\|f_{ij}\|=\|f_i\|_\infty=1.$ For $1<p<\infty,$ from the equality condition of H\"{o}lder's inequality in (\ref{eq-09}), we have $$\|f_{ij}\|^q=\lambda \|T_j^0x_i\|^p=\lambda \|T_j^0\|^p$$ for all $1\leq j\leq d$ and for some non-zero constant $\lambda.$  	If for some $1\leq j\leq d,~\|T_j^0\|=0,$ then $T_j=z_j^0S_j$ and so $dist(T_j,\mathbb{F}S_j)=0,$ that is (\ref{eq-02}) holds trivially. So assume that $\|T_j^0\|\neq 0,$ that is $\|f_{ij}\|\neq 0.$
	Now for all $\textbf{z}=(z_1,\ldots z_d)\in \mathbb{F}^d,$
	\begin{eqnarray*}
		&&\sum_{i=1}^h t_if_i\otimes x_i(\textbf{z}\mathcal{S})=0\\
		&\Rightarrow & \sum_{i=1}^h t_if_i(\textbf{z}\mathcal{S}x_i)=0\\
		&\Rightarrow& \sum_{i=1}^h t_i\sum_{j=1}^d f_{ij} (z_jS_jx_i)=0\\
		&\Rightarrow & \sum_{j=1}^d z_j\bigg( \sum_{i=1}^h t_i f_{ij}(S_jx_i)\bigg)=0\\
		&\Rightarrow & \sum_{i=1}^h t_i f_{ij}(S_jx_i)=0, \quad \forall ~1\leq j\leq d.\\
		&\Rightarrow& \sum_{i=1}^h t_i \frac{1}{\|f_{ij}\|}f_{ij}\otimes x_i(S_j)=0 .
	\end{eqnarray*}
 Now  from Theorem \ref{th-singer}, it follows that $z_j^0S_j\in P_{\mathbb{F}S_j}(T_j),$ equivalently
$$dist(T_j,\mathbb{F}S_j)=\|T_j^0\|.$$ Therefore,
$$dist(\mathcal{T},\mathbb{F}^d\mathcal{S})^p=\|\mathcal{T}^0\|^p=\sum_{j=1}^d\|T_j^0\|^p=\sum_{j=1}^ddist(T_j,\mathbb{F}S_j)^p.  \eqno \qedhere $$
\end{proof}

Now, we exhibit examples of tuples satisfying the condition of Theorem \ref{th-02}.

\begin{example}\label{ex-01}
(a) Consider a reflexive Banach space $\mathcal{X}.$ Choose $B_j,A\in \mathcal{K}(\mathcal{X},\mathcal{Y}).$ Suppose $x\in M_{A}.$ Let $H$ be a hyperspace such that $x\perp_B H.$ Define
\begin{align*}
T_j(\alpha x+h)&=\alpha Ax+\frac{1}{j+1}Ah,\\
S_j(\alpha x+h)&=B_j(h), \quad\text{ where }\alpha\in \mathbb{F}, h\in H.
 \end{align*}
If is easy to see that $S_j,T_j$ are compact and $x\in M_{T_j}$ for all $j\in \mathbb{N}.$ Indeed,
\begin{eqnarray*}
&&\|T_j(\alpha x+h)\|\\
&=& \|\frac{1}{j+1}A(\alpha x +h)+(1-\frac{1}{j+1})A(\alpha x)\|	\\
&\leq&\frac{1}{j+1}\|A(\alpha x +h)\|+(1-\frac{1}{j+1})\|A(\alpha x)\|	\\
&\leq&\frac{1}{j+1}\|A\|\|(\alpha x +h)\|+(1-\frac{1}{j+1})\|A\|\|\alpha x\|	\\
&\leq&\frac{1}{j+1}\|A\|\|(\alpha x +h)\|+(1-\frac{1}{j+1})\|A\|\|(\alpha x+h)\|,~\text{since } x\perp_B h	\\
&=&\|A\|\|(\alpha x +h)\|.
\end{eqnarray*}
Thus $\|T_j\|\leq\|A\|.$
Moreover, $\|T_j(x)\|=\|Ax\|=\|A\|,$ that is, $x\in \cap_{j=1}^dM_{T_j}.$  On the other hand, since $T_j$ is a convex combination of two compact operators, so $T_j$ is compact.  To prove the compactness of $S_j,$ suppose $C_j(\alpha x+h)=\alpha B_jx,$ where $\alpha\in \mathbb{F}$,  and $h\in H.$ Then $$\|C_j(\alpha x+h)\|\leq\|B_j\|\|\alpha x\|\leq \|B_j\|\|\alpha x+h\|$$ implies that $C_j$ is compact, and so $S_j=B_j-C_j$ is compact. 
Now, assume $\mathcal{T}=(T_1,\ldots,T_j), \mathcal{S}=(S_1,\ldots,S_d).$ Observe that, for all $\textbf{z}=(z_1,\ldots,z_d)\in \mathbb{F}^d,$
\[\|\mathcal{T}-\textbf{z}\mathcal{S}\|^p\geq \|\mathcal{T}x-\textbf{z}\mathcal{S}x\|^p=\sum_{j=1}^d\|T_jx-z_jS_jx\|^p=\sum_{j=1}^d\|T_jx\|^p=\sum_{j=1}^d\|T_j\|^p.\] Thus, 
 $$dist(\mathcal{T},\mathbb{F}^d\mathcal{S})^p\geq \sum_{j=1}^d\|T_j\|^p =\|\mathcal{T}\|^p,$$ where the last equality follows from Lemma \ref{lem-02}. Therefore, $ dist(\mathcal{T},\mathbb{F}^d\mathcal{S})^p=\|\mathcal{T}\|^p,$ and so (\ref{eq-07}) holds.\\
 
 (b) Choose $\Lambda_j=(1,\lambda_{j2},\lambda_{j3},\ldots)\in c_0$	such that $\|\Lambda_j\|_\infty=1.$ Let $1<m<\infty.$ Define $T_j:\ell_m\to\ell_m,$ and $S_j:\ell_m\to\ell_m,$ as follows:
 \begin{align*}
 T_j(x_1,x_2,x_3,\ldots)&=(x_1,\lambda_{j2}x_2,\lambda_{j3}x_3,\ldots)\\
 S_j(x_1,x_2,x_3,\ldots)&=(0,\lambda_{j2}x_2,\lambda_{j3}x_3,\ldots).
 \end{align*}
 Then it is easy to observe that $T_j,S_j$ are compact, $\|T_j\|=\|\Lambda_j\|_\infty=1,$ and $e_1\in \cap_{j=1}^d M_{T_j}.$ Now, proceeding as the previous example, we see (\ref{eq-07}) holds for $\mathcal{T}=(T_1,T_2,\ldots,T_d)$ and $\mathcal{S}=(S_1,S_2,\ldots,S_d).$
\end{example}

The next theorem provides a relation between the G\^ateaux derivatives of tuples of operators and that of its components.
\begin{theorem}\label{th-07}
	Suppose $\mathcal{X}$ is reflexive. Let $\mathcal{T},\mathcal{S}\in \mathcal{K}(\mathcal{X},\ell_{p}^d(\mathcal{Y})),$ where $\cap_{i=1}^dM_{T_i}\neq \emptyset.$ Then 
	\begin{align*}
	\sum_{i=1}^d\rho_-(T_i,S_i)\leq 	\rho_-(\mathcal{T},\mathcal{S})\leq 	\rho_+(\mathcal{T},\mathcal{S})\leq\sum_{i=1}^d\rho_+(T_i,S_i).
	\end{align*}
\end{theorem}
\begin{proof}
Suppose $\phi\in E_{J(\mathcal{T})}.$ Then $\phi\in E_{\mathcal{K}(\mathcal{X},\ell_{p}^d(\mathcal{Y}))^*}.$ So there exists $x\in E_{\mathcal{X}}$ and $f=(f_1,\ldots,f_d)\in E_{\ell_{q}^d(\mathcal{Y}^*)}$	such that $\phi=f\otimes x.$ Now proceeding similarly as (\ref{eq-09}) and (\ref{eq-08}), we get $f_i\otimes x(T_i)=\|T_i\|,$ that is $f_i\otimes x\in J(T_i)$ for all $1\leq i\leq d.$ Moreover, $f_i\otimes x\in E_{\mathcal{K}(\mathcal{X},\mathcal{Y})^*}.$ Now, 
\begin{eqnarray*}
	\phi(\mathcal{S})&=& f(\mathcal{S}x)\\
	&=&\sum_{i=1}^df_i\otimes x(S_i)\\
	&\leq & \sum_{i=1}^d \rho_+(T_i,S_i)\\
	\Rightarrow \rho_+(\mathcal{T},\mathcal{S})	&\leq & \sum_{i=1}^d \rho_+(T_i,S_i).
	\end{eqnarray*} 
Similarly, we get $\sum_{i=1}^d\rho_-(T_i,S_i)\leq 	\rho_-(\mathcal{T},\mathcal{S}).$
\end{proof}

Note that Example \ref{ex-01} provides us examples of tuples satisfying the condition of Theorem \ref{th-07}. 
In Proposition \ref{prop-04}, we proved the distance formula assuming smoothness of a tuple of operators. As an application of the last theorem, we immediately get a sufficient condition for the smoothness of a tuple of operators in terms of its components.
\begin{cor}\label{cor-03}
	Suppose $\mathcal{X}$ is reflexive. Let $\mathcal{T}\in \mathcal{K}(\mathcal{X},\ell_{p}^d(\mathcal{Y})),$ where $\cap_{i=1}^dM_{T_i}\neq \emptyset.$  If $T_i$ is smooth for each $1\leq i\leq d,$ then $\mathcal{T}$ is smooth.
\end{cor}
\begin{proof}
	Since each $T_i $ is smooth, $\rho_+(T_i,S_i)=\rho_-(T_i,S_i)$ for all $S_i\in \mathcal{K}(\mathcal{X},\mathcal{Y}).$  Hence from Theorem \ref{th-07} it follows that $	\rho_-(\mathcal{T},\mathcal{S})=\rho_+(\mathcal{T},\mathcal{S})$ for all $\mathcal{S}\in \mathcal{K}(\mathcal{X},\ell_{p}^d(\mathcal{Y})),$  and so we conclude that $\mathcal{T}$ is smooth. 
\end{proof}

\begin{remark}
	In Example \ref{ex-01}, if we additionally assume that $\mathcal{X}$ is strictly convex, and $\mathcal{Y}$ is smooth, then it is easy to check that $\cap_{j=1}^dM_{T_j}=\{\alpha x:|\alpha|=1\}.$ Therefore, each $T_j$ is smooth, and so the tuple $\mathcal{T}$ is smooth by Corollary \ref{cor-03}.
\end{remark}

The converse of the above corollary is not true. The following example shows that there exists a smooth operator $\mathcal{T}\in \mathcal{K}(\mathcal{X},\ell_2^d(\mathcal{Y}))$ such that $\cap_{i=1}^dM_{T_i}\neq \emptyset.$ However, $T_i\in\mathcal{K}(\mathcal{X},\mathcal{Y})$ is not smooth  for each $1\leq i\leq d.$ 

\begin{example}
	Suppose $1<m<\infty$ and $\mathcal{X}=\mathcal{Y}=\ell_m^d.$ For $1\leq n\leq d,$ define $T_n:\ell_m^d\to \ell_m^d$ as follows.
	\[T_n(x)=\Big(x_1,\frac{x_2}{n+1},\ldots,\frac{x_{n-1}}{n+1},x_n,\frac{x_{n+1}}{n+1},\ldots,\frac{x_{d}}{n+1}\Big),\] 
    where $x=(x_1,x_2,\ldots,x_d)\in \ell_m^d.$
	Clearly, $\|T_n\|=1$ and $ M_{T_n}=span\{e_1,e_n\}\cap S_{\mathcal{X}}$ for all $1\leq n\leq d,$ where $e_n$ is the standard co-ordinate vector. So $T_n$ is not smooth for any $1\leq n\leq d.$ 
	Let $\mathcal{T}=(T_1,T_2,\ldots,T_d)\in \mathcal{K}(\mathcal{X},\ell_p^d(\mathcal{Y})).$ Since $\cap_{n=1}^dM_{T_n}=\{\alpha e_1:\alpha\in \mathbb{F},|\alpha|=1\},$ by Lemma \ref{lem-02}, $$M_{\mathcal{T}}=\{\alpha e_1:\alpha\in \mathbb{F},|\alpha|=1\} \text{ and } \|\mathcal{T}\|^p=\sum_{n=1}^d\|T_n\|^p=d.$$ We claim that $\mathcal{T}$ is smooth. Let $f\in J(\mathcal{T}e_1).$ Then $f\in (\ell_p^d(\mathcal{Y}))^*=\ell_q^d(\mathcal{Y}^*),$ where $\frac{1}{p}+\frac{1}{q}=1,$ that is, $f=(f_1,\ldots,f_d),$ where $f_n\in \mathcal{Y}^*$ for all $1\leq n\leq d.$ Note that, $\|\mathcal{T}\|=\|\mathcal{T}e_1\|=f(\mathcal{T}e_1).$ Now, proceeding similarly as (\ref{eq-09}) and (\ref{eq-08}), we get 
	\[f_i(T_ie_1)=\|f_i\|\|T_ie_1\|, ~\|f_i\|\neq 0, ~\forall~1\leq i\leq d,\]
	which implies that $\frac{1}{\|f_i\|}f_i\in J(T_ie_1).$
	Since $\ell_m^d$ is smooth space, $T_ie_1$ is smooth and so $J(T_ie_1)=\{\frac{1}{\|f_i\|}f_i\}$ for all $1\leq i\leq d.$ Thus, $J(\mathcal{T}e_1)$ is singleton and so $\mathcal{T}e_1$ is smooth, consequently $\mathcal{T}$ is smooth.
\end{example}

The next corollary  provides a sufficient condition for the equivalence $	\mathcal{T}\perp_B \mathbb{F}^d\mathcal{S}  \Leftrightarrow T_j\perp_B S_j$  for each $ 1\leq j\leq d.$
\begin{cor}\label{cor-01}
	Suppose $\mathcal{X}$ is reflexive and $\mathcal{T},\mathcal{S}\in \mathcal{K}(\mathcal{X},\ell_p^d(\mathcal{Y})).$  Suppose $\cap_{i=1}^dM_{T_i}\neq \emptyset.$ Then
	\begin{equation}\label{eq-06}
		\mathcal{T}\perp_B \mathbb{F}^d\mathcal{S} \quad \Rightarrow \quad T_j\perp_B S_j \text{ for each } 1\leq j\leq d.
		\end{equation} Moreover, if each $T_j$ is smooth, then 
	\[T_j\perp_B S_j \text{ for all } 1\leq j\leq d \quad \Rightarrow \quad \mathcal{T}\perp_B \mathbb{F}^d\mathcal{S} .\]
\end{cor}
\begin{proof}
Note that if $\mathcal{T}\in \mathbb{F}^d\mathcal{S},$ then the result follows trivially.  So consider that $\mathcal{T}\notin \mathbb{F}^d\mathcal{S}.$ First assume	$\mathcal{T}\perp_B \mathbb{F}^d\mathcal{S},$ that is, $dist(\mathcal{T},\mathbb{F}^d\mathcal{S})=\|\mathcal{T}\|.$ Now, by Theorem \ref{th-02}, for each $1\leq j\leq d,$ $dist(T_j,\mathbb{F}S_j)=\|T_j\|$ and so $T_j\perp_B S_j.$ \\
Conversely, assume each $T_j$ is smooth. Then $M_{T_i}=\{\alpha x_i:\alpha\in \mathbb{F},|\alpha |=1\}$ for some $x_i\in S_{\mathcal{X}}$ and $T_ix_i$ is smooth in $\mathcal{Y}.$
From $\cap_{i=1}^dM_{T_i}\neq \emptyset,$ we get $x_i=\beta_ix,$ say, for all $1\leq i\leq d,$ where $\beta_i\in \mathbb{F},$  $|\beta_i|=1,$ and $x\in S_{\mathcal{X}}.$  Therefore, $\cap_{i=1}^dM_{T_i}=\{\alpha x:\alpha\in \mathbb{F},|\alpha|=1\}.$ By Lemma \ref{lem-02}, $x\in M_{\mathcal{T}}.$ Now from \cite[pp. 141]{MPS}, it follows that  $T_j\perp_B S_j$ implies  $T_jx\perp_B S_jx.$ Let $\textbf{z}=(z_1,\ldots,z_d)\in \mathbb{F}^d.$ Then 
\[\|\mathcal{T}-\textbf{z}\mathcal{S}\|^p\geq \|\mathcal{T}x-\textbf{z}\mathcal{S}x\|^p=\sum_{j=1}^d\|T_jx-z_jS_jx\|^p\geq \sum_{j=1}^d\|T_jx\|^p =\|\mathcal{T}x\|^p=\|\mathcal{T}\|^p. \]
Therefore, $\mathcal{T}\perp_B \mathbb{F}^d\mathcal{S}.$
\end{proof}

We end the discussion on the case $1\leq p<\infty$ with an interesting example. It shows that if $\cap_{i=1}^d{M_{T_j^0}}= \emptyset,$ then (\ref{eq-02}) and (\ref{eq-06}) may not hold but  (\ref{eq-07}) may hold.
\begin{example}
	Consider $\mathbb{F}=\mathbb{R}.$ Define $T_1,T_2,S_1,S_2:\ell_2^2\to\ell_2^2$ as follows.
	\begin{align*}
	&T_1(a,b)=\Big(\frac{a}{2},b\Big),\quad T_2(a,b)=\Big(a,\frac{b}{2}\Big),\quad \forall~(a,b)\in \ell_2^2,\\ 
&S_1(a,b)=\Big(\frac{a-b}{2},\frac{a-b}{2}\Big)=-S_2(a,b),\quad \forall~(a,b)\in \ell_2^2.
\end{align*} 
Let $\mathcal{T}=(T_1,T_2), \mathcal{S}=(S_1,S_2)\in \mathcal{K}(\ell_2^2,\ell_2^2(\ell_2^2)).$ Clearly, $\mathcal{T}\notin \mathbb{F}^2\mathcal{S}.$ Note that $$\langle T_1(1,1),S_1(1,1)\rangle =0=\langle T_2(1,1),S_2(1,1)\rangle.$$ Since $(\frac{1}{\sqrt{2}},\frac{1}{\sqrt{2}})\in M_{\mathcal{T}},$ so by \cite[Th. 2.3]{GS}, $\mathcal{T}\perp_B \mathbb{F}^2\mathcal{S},$ that is, $$dist(\mathcal{T},\mathbb{R}^2\mathcal{S})=\|\mathcal{T}\|=\frac{\sqrt{5}}{2}.$$ (Note that  \cite[Th. 2.3]{GS} is proved for $\mathbb{F}=\mathbb{C}.$ However, the same proof also holds for $\mathbb{F}=\mathbb{R}.$)
	Now it is easy to check that $M_{T_1}=\{\pm(0,1)\},$ $M_{T_2}=\{\pm (1,0)\}.$ Clearly, $M_{T_1}\cap M_{T_2}=\emptyset.$ Since $T_1(0,1)\not\perp S_1(0,1)$ and $T_2(1,0)\not\perp S_2(1,0),$ so by \cite[Th. 4.2.2]{MPS}, $$T_1\not\perp_B S_1 \text{ and } T_2\not\perp_B S_2,$$
	which implies that $$dist(T_1,\mathbb{R}S_1)\neq \|T_1\|, \text{ and } dist(T_2,\mathbb{R}S_2)\neq \|T_2\|.$$ This proves that here (\ref{eq-02}) and (\ref{eq-06}) are not true. However, we show that (\ref{eq-07}) holds. Note that 
	\[\|(T_1+\frac{1}{2}S_1)(a,b)\|^2=\|(T_2+\frac{1}{2}S_2)(a,b)\|^2=\frac{5}{8}, \quad \forall~(a,b)\in S_{\ell_2^2},\] and 
	$S_1(1,1)=S_2(1,1)=0.$ Since $(\frac{1}{\sqrt{2}},\frac{1}{\sqrt{2}})\in M_{T_i+\frac{1}{2}S_i}$ for $i=1,2,$ from \cite[Th. 4.2.2]{MPS} it follows that 
	\[T_1+\frac{1}{2}S_1\perp_B S_1,\quad T_2+\frac{1}{2}S_2\perp_B S_2.\]
	Equivalently, 
	\[dist(T_1,\mathbb{R}S_1)^2=\|T_1+\frac{1}{2}S_1\|^2=\frac{5}{8},\] and \[dist(T_2,\mathbb{R}S_2)^2=\|T_2+\frac{1}{2}S_2\|^2=\frac{5}{8}.\]
	So in this case,
	\[dist(\mathcal{T},\mathbb{R}^2\mathcal{S})^2=dist(T_1,\mathbb{R}S_1)^2+dist(T_2,\mathbb{R}S_2)^2.\] 
\end{example}

Now, we turn our attention to the simpler case $p=\infty.$ 

\begin{theorem}\label{th-01}
	Suppose $\mathcal{T},\mathcal{S}\in \mathcal{L}(\mathcal{X},\ell_{\infty}^d(\mathcal{Y}_k)).$ Then 
	\[\|\mathcal{T}\|=\max_{1\leq i\leq d}\|T_i\|, \text{ and }\]
	\[dist(\mathcal{T},\mathbb{F}^d\mathcal{S})=\max_{1\leq i\leq d}dist(T_i,\mathbb{F}S_i).\]
\end{theorem}
\begin{proof}
	Since $\|T_ix\|\leq \|\mathcal{T}x\|$ for each $1\leq i\leq d$ and $x\in S_{\mathcal{X}},$ so $\|T_k\|\leq \|\mathcal{T}\|,$ that is, $\max_{1\leq i\leq d}\|T_i\|\leq \|\mathcal{T}\|.$ For the reverse inequality, observe that
	\begin{align*}
		\|\mathcal{T}x\|=\max_{1\leq i\leq d}\|T_ix\|&\leq\max_{1\leq i\leq d}\|T_i\|, \quad\text{ for all } x\in S_{\mathcal{X}}\\
		\Rightarrow \|\mathcal{T}\|&\leq\max_{1\leq i\leq d}\|T_i\|.
	\end{align*} 

Since $\mathbb{F}^d\mathcal{S}$ is the $\ell_\infty $ direct sum of $\mathbb{F}S_i,$ $(1\leq i\leq d)$ so from \cite{I}, it follows that 	$$dist(\mathcal{T},\mathbb{F}^d\mathcal{S})=\max_{1\leq i\leq d}dist(T_i,\mathbb{F}S_i). \eqno\qedhere$$

\end{proof}

The next corollary provides the equivalence of B-J orthogonality of tuples of operators and B-J orthogonality of its components.
\begin{cor}\label{cor-04}
	Suppose  $\mathcal{T},\mathcal{S}\in \mathcal{L}(\mathcal{X},\ell_{\infty}^d(\mathcal{Y}_k)).$ Then 
	\[\mathcal{T}\perp_B \mathbb{F}^d\mathcal{S}\quad \Leftrightarrow \quad T_i\perp_B S_i~\text{and } \|T_i\|=\|\mathcal{T}\|, \text{ for some } 1\leq i\leq d.\]
\end{cor}
\begin{proof}
	First assume that $\mathcal{T}\perp_B \mathbb{F}^d\mathcal{S}.$ Then $dist(\mathcal{T},\mathbb{F}^d\mathcal{S})=\|\mathcal{T}\|.$ By Theorem \ref{th-01}, we get for some $1\leq i\leq d,$
	\[dist(T_{i},\mathbb{F}S_{i})= \|T_{i}\|=\|\mathcal{T}\|=dist(\mathcal{T},\mathbb{F}^d\mathcal{S}),\] which implies that $T_i\perp_B S_i.$\\
	On the other hand, suppose $\|\mathcal{T}\|=\|T_i\|$ for some $1\leq i\leq d$ and $T_i\perp_B S_i.$ Then
	\begin{eqnarray*}
		&& dist(T_i,\mathbb{F}S_i)=\|T_i\|=\|\mathcal{T}\|\\
		& \Rightarrow &\max_{1\leq j\leq d} dist(T_j,\mathbb{F}S_j)\geq \|\mathcal{T}\|\\
		&\Rightarrow& dist(\mathcal{T},\mathbb{F}^d\mathcal{S})\geq \|\mathcal{T}\|, ~\text{ (using Theorem \ref{th-01})}. 
	\end{eqnarray*}
	Since $\|\mathcal{T}\|\geq dist(\mathcal{T},\mathbb{F}^d\mathcal{S})$ holds trivially, so we have $dist(\mathcal{T},\mathbb{F}^d\mathcal{S})= \|\mathcal{T}\|,$ that is, $\mathcal{T}\perp_B \mathbb{F}^d\mathcal{S}.$
\end{proof}

Note that, from the last corollary it follows that if $\|T_i\|=\|T_j\|$ for all $1\leq i,j\leq d,$  then $\mathcal{T}\perp_B \mathbb{F}^d\mathcal{S}\Leftrightarrow  T_i\perp_B S_i$ for some $1\leq i\leq d.$
Next, we obtain an interesting distance formula and equivalence of B-J orthogonality in $\mathcal{K}(\mathcal{X},\ell_\infty^d),$ using the previous results. For a detailed study of distance formula in the space of compact operators, the readers are invited to look into \cite{MP22}. Recall that $\mathcal{X}$ is said to be strictly convex, if $S_{\mathcal{X}}=E_{\mathcal{X}}.$ 
\begin{cor}
	Suppose $\mathcal{X}$ is a reflexive and strictly convex Banach space. Let $T,S\in \mathcal{K}(\mathcal{X},\ell_\infty^d)$ and $T=(f_1,\ldots,f_d), S=(g_1,\ldots, g_d),$ where $f_i,g_i\in \mathcal{X}^*,$ $1\leq i\leq d.$ Then the following are true.\\
	\rm(i) $$dist(T,\mathbb{F}^dS)=\underset{1\leq i\leq d}{\max}\|f_i|_{\ker(g_i)}\|.$$
	\rm(ii) If $\|T\|\neq 0,$ then $T\perp_B \mathbb{F}^dS$ implies that for some $1\leq i\leq d,$ $g_i(x)=0,$ where $x\in M_{f_i}.$ \\
	\rm(iii) If $\|f_i\|=\lambda (\neq 0)$ for all $1\leq i\leq d,$ then $T\perp_B \mathbb{F}^dS$ if and only if for some $1\leq i\leq d,$ $g_i(x)=0,$ where $x\in M_{f_i}.$
\end{cor}
\begin{proof}
	(i) From \cite[Cor. 3.9]{MP22}, we have $$dist(f_i,\mathbb{F}g_i)=\max\{|f_i(x)|:x\in S_{\mathcal{X}}, g_i(x)=0\},$$ which is clearly equal to $\|f_i|_{\ker(g_i)}\|.$ Now, the result immediately follows from Theorem \ref{th-01}.\\
	
	(ii) Suppose $T\perp_B \mathbb{F}^dS.$ Then from the last corollary, we have $f_i\perp_B g_i$ for some $1\leq i\leq d$ and $\|T\|=\|f_i\|.$ Since $\mathcal{X}$ is reflexive and strictly convex, so $\mathcal{X}^*$ is smooth. Therefore, $f_i$ is smooth and so $M_{f_i}=\{\alpha x:\alpha \in \mathbb{F}, |\alpha|=1\},$ for some $x\in S_{\mathcal{X}}.$ Now from \cite[Th. 2.13]{SPM2}, it follows that $f_i(x)g_i(x)=0.$ Since $x\in M_{f_i}$ and $\|f_i\|=\|T\|\neq 0,$ we must have $f_i(x)\neq 0,$ that is $g_i(x)=0.$\\
	
	(iii) Since for all $1\leq i\leq d,$ $\|f_i\|=\lambda,$ so $\|T\|=\|f_i\|.$ Hence,  from Corollary \ref{cor-04}, it follows that $T\perp_B S$ if and only if $f_i\perp_B g_i$ for some $i.$ Now as in (ii), this is equivalent to $g_i(x)=0,$ where $x\in M_{f_i}.$
\end{proof}

We end the section with the relation between the G\^ateaux derivatives of tuples of operators and the same of its components.
\begin{theorem}\label{th-06}
	Suppose $\mathcal{T},\mathcal{S}\in \mathcal{L}(\mathcal{X},\ell_\infty^d(\mathcal{Y}_k)).$ Then 
	\begin{align*}
		\rho_+(\mathcal{T},\mathcal{S})&=\max\{\rho_+(T_i,S_i):1\leq i\leq d, \|T_i\|=\|\mathcal{T}\|\},\\
		\rho_-(\mathcal{T},\mathcal{S})&=\min\{\rho_-(T_i,S_i):1\leq i\leq d, \|T_i\|=\|\mathcal{T}\|\}.
	\end{align*}	
\end{theorem}
\begin{proof}
Note that, 
\begin{eqnarray*}
	\rho_+(\mathcal{T},\mathcal{S})&=& \lim_{t\to 0+}\frac{\|\mathcal{T}+t\mathcal{S}\|-\|\mathcal{T}\|}{t}\\
	&=& \lim_{t\to 0+}\frac{\|\mathcal{T}+t\mathcal{S}\|-\|T_i\|}{t}, \quad \text{ where }\|\mathcal{T}\|=\|T_i\|\\
	&\geq& \lim_{t\to 0+}\frac{\|T_i+tS_i\|-\|T_i\|}{t}\\
	&=&\rho_+(T_i,S_i)\\
	\Rightarrow 	\rho_+(\mathcal{T},\mathcal{S})&\geq& \max\{\rho_+(T_i,S_i):1\leq i\leq d, \|T_i\|=\|\mathcal{T}\|\}.
	\end{eqnarray*}
For the reverse inequality, by Theorem \ref{th-01}, assume that $$\|\mathcal{T}+\frac{1}{n}\mathcal{S}\|=\|T_{i(n)}+\frac{1}{n}S_{i(n)}\|,\quad\text{where } i(n)\in \{1,2,\ldots,d\}.$$ There exits a subsequence, say $\{n_m\}$ such that $i(n_m)$ is constant. Assume $i(n_m)=j$ for all $m\in \mathbb{N},$ that is, $\|\mathcal{T}+\frac{1}{n_m}\mathcal{S}\|=\|T_j+\frac{1}{n_m}S_j\|.$ Taking limit $m\to \infty,$ we get $\|\mathcal{T}\|=\|T_j\|.$ Therefore,
\begin{eqnarray*}
\rho_+(\mathcal{T},\mathcal{S})&=&\lim_{m\to\infty}	\frac{\|\mathcal{T}+\frac{1}{n_m}\mathcal{S}\|-\|\mathcal{T}\|}{\frac{1}{n_m}}\\
&=&\lim_{m\to\infty}	\frac{\|T_j+\frac{1}{n_m}S_j\|-\|T_j\|}{\frac{1}{n_m}}\\
&=&\rho_+(T_j,S_j)\\
&\leq& \max\{\rho_+(T_i,S_i):1\leq i\leq d, \|T_i\|=\|\mathcal{T}\|\}.
\end{eqnarray*}

	The other part of the proof follows similarly.
\end{proof}

\section{Acknowledgement}
The author would like to thank DST, Govt. of India for the financial support in the form of INSPIRE Faculty Fellowship (DST/INSPIRE/04/2022/001207). The author is thankful to Dr. Krishna Kumar Gupta, Prof. Priyanka Grover, and Dr. Susmita Seal for suggesting an alternative proof of Theorem \ref{th-01}.

\bibliographystyle{amsplain}

\begin{thebibliography}{99}
	
		
		\bibitem{AGKZ}  L. Aramba\v{s}i\'c, A. Guterman, B. Kuzma and S. Zhilina, \textit{Birkhoff-James orthogonality: characterizations, preservers, and orthogonality graphs}, Trends Math.,
		Birkhäuser, Cham, (2022), 255–302,
		ISBN: 978-3-031-02103-9.
		
	

\bibitem{BS} R. Bhatia and P. \v{S}emrl, \textit{Orthogonality of matrices and some distance problems}, Linear Algebra Appl., \textbf{287} (1999), 77-85. 
	
\bibitem{B} G. Birkhoff,
\textit{Orthogonality in linear metric spaces},
Duke Math. J., \textbf{1} (1935), 169-172.

\bibitem{C} J. Chmieli\'nski, \textit{Approximate Birkhoff-James orthogonality in normed linear spaces and related topics}, Trends Math., 
Birkhäuser/Springer, Cham, (2022), 303–320.

\bibitem{CL} W. Cheney, and W. Light, \textit{A course in approximation theory}, Grad. Stud. Math., \textbf{101}
American Mathematical Society, Providence, RI, (2009), xvi+359 pp.

\bibitem{DGZ} R. Deville, G. Godefroy, and V. Zizler, \textit{Smoothness and renormings in Banach spaces}, Pitman Monogr. Surveys Pure Appl. Math., 64,  New York, (1993), xii+376 pp. 

\bibitem{GI} G. Godefroy, and V. Indumathi, \textit{Strong proximinality and polyhedral spaces}, Rev. Mat. Complut., \textbf{14} (2001), No. 1, 105-125.

\bibitem{G} P. Grover, \textit{Orthogonality to matrix subspaces, and a distance formula}, Linear Algebra Appl., \textbf{445} (2014), 280-288.

\bibitem{GS} P. Grover and S. Singla, \textit{Subdifferential set of the joint numerical radius of a tuple of matrices}, Linear Multilinear Algebra, \textbf{71}  (2023), No. 17, 2709-2718.


\bibitem{GS2} P. Grover and S. Singla, \textit{A distance formula for tuples of operators}, Linear Algebra Appl., \textbf{650} (2022), 267-285.

\bibitem{GS3} P. Grover and S. Singla, \textit{Birkhoff-James orthogonality and applications: a survey}, Oper. Theory Adv. Appl., \textbf{282} (2021),
Birkhäuser/Springer, Cham,  293-315.

\bibitem{HL} J.B. Hiriart-Urruty, C. Lemarechal, \textit{Fundamentals of Convex Analysis}, Springer, Berlin, (2001), x+259 pp.


\bibitem{I} V. Indumathi, \textit{Semi-continuity of metric projections in $\ell_\infty$-direct sums}, Proc. Amer. Math. Soc., \textbf{133 } (2005), No. 5, 1441-1449.

\bibitem{J}  R. C. James,  \textit{Orthogonality and linear functionals in normed linear spaces}, Trans. Amer. Math. Soc., \textbf{61} (1947), 265-292.

	

	
	\bibitem{M}	A. Mal, \textit{An approximation problem in the space of bounded operators}, Numer. Funct. Anal. Optim., \textbf{44} (2023), No. 2, 124-137.
		
\bibitem{MP22} A. Mal and K. Paul, \textit{Distance formulae and best approximation in the space of compact operators}, J. Math. Anal. Appl., \textbf{509} (2022), Paper No. 125952, 19 pp.
		
\bibitem{MP} A. Mal and K. Paul, \textit{Birkhoff-James orthogonality to a subspace of operators defined between Banach spaces}, J. Operator Theory, \textbf{85} (2021), No. 2, 463-474.

\bibitem{MPS} A. Mal, K. Paul and D. Sain, \textit{Birkhoff-James orthogonality and geometry of operator spaces}, Infosys Sci. Found. Ser. Math. Sci., Springer, (2024), xiii+251 pp., ISBN 978-981-99-7110-7.
		

\bibitem{RAO21}  T.S.S.R.K. Rao, \textit{Operators Birkhoff-James orthogonal to spaces of operators}, Numer. Funct. Anal. Optim., \textbf{42} (2021), No. 10, 1201-1208. 
	
		
	\bibitem{RS} W. M. Ruess and C. P. Stegall, \textit{Extreme points in duals of operator spaces}, Math. Ann., \textbf{261} (1982), 535-546.
			
	\bibitem{Sb} D. Sain, \textit{On best approximations to compact operators}, Proc. Amer. Math. Soc., \textbf{149} (2021), No. 10, 4273-4286. 
			
\bibitem{SMP} D. Sain, A. Mal and K. Paul, \textit{Some remarks on Birkhoff-James orthogonality of linear operators}, Expo. Math., \textbf{38} (2020), 138-147.
				
	\bibitem{SPM2} D. Sain, K. Paul and A. Mal, \textit{A complete characterization of Birkhoff-James orthogonality in infinite dimensional normed space}, J. Operator Theory, \textbf{ 80} (2018), 399-413.
	
	

	\bibitem{S} I. Singer, \textit{Best approximation in normed linear spaces by elements of linear subspaces}, Grundlehren Math. Wiss., \textbf{171} (1970), Springer-Verlag, Berlin, Heidelberg, New York.
	
	
	\end{thebibliography}

\end{document}